\documentclass[preprint,12pt]{elsarticle}

\usepackage{amssymb}
\usepackage{amsmath}
\usepackage{bm}
\usepackage[T1]{fontenc}
\usepackage[latin9]{inputenc}
\usepackage{xcolor}
\usepackage{pdfcolmk}
\usepackage{babel}
\usepackage{bm}
\usepackage{amsthm}
\usepackage{comment}
\usepackage{hyperref}

\usepackage{color,bbm,pdfsync}
\theoremstyle{plain}
\newtheorem{thm}{\protect\theoremname}[section]
\theoremstyle{definition}
\newtheorem{example}[thm]{\protect\examplename}
\theoremstyle{definition}
\newtheorem{defn}[thm]{\protect\definitionname}
\theoremstyle{plain}
\newtheorem{cor}[thm]{\protect\corollaryname}
\theoremstyle{remark}
\newtheorem{rem}[thm]{\protect\remarkname}
\newtheorem{lem}[thm]{\protect\lemmaname}
\newtheorem{prop}[thm]{\protect\propositionname}

\providecommand{\corollaryname}{Corollary}
\providecommand{\definitionname}{Definition}
\providecommand{\examplename}{Example}
\providecommand{\remarkname}{Remark}
\providecommand{\theoremname}{Theorem}
\providecommand{\lemmaname}{Lemma}
\providecommand{\propositionname}{Proposition}

\journal{Stochastic Processes and Their Applications}

\begin{document}

\begin{frontmatter}



\title{Stochastic Currents of Fractional Brownian Motion: Existence and Regularity}


\author{Martin Grothaus} 

\affiliation{organization={Department of Mathematics, RPTU Kaiserslautern-Landau},
            city={Kaiserslautern},
            postcode={67663}, 
            country={Germany}}
						
\author{Jos{\'e} Lu{\'i}s da Silva} 

\affiliation{organization={CIMA, University of Madeira, Campus da Penteada},
            city={Funchal},
            postcode={9020-105}, 
            country={Portugal}}
						
\author{Herry Pribawanto Suryawan} 

\affiliation{organization={Department of Mathematics, Sanata Dharma University},
            city={Yogyakarta},
            postcode={55281}, 
            country={Indonesia}}	

\author{Thomas Ullrich} 

\affiliation{organization={Department of Mathematics, RPTU Kaiserslautern-Landau},
            city={Kaiserslautern},
            postcode={67663}, 
            country={Germany}}            

\begin{abstract}
By using white noise analysis, we study the integral kernel $\xi(x)$,
$x\in\mathbb{R}^{d}$, of stochastic currents corresponding to fractional Brownian motion with Hurst parameter
$H\in(0,1)$. For $x\in\mathbb{R}^{d}\backslash\{0\}$ and $d\ge1$  we show that
the kernel $\xi(x)$ is well-defined as a Hida distribution for all $H\in(0,1)$.  For $x=0$ and $d=1$, $\xi(0)$ is a Hida distribution for all $H\in(0,1)$. For $d\ge2$, then $\xi(0)$ is a Hida distribution only for $H\in(0,1/d)$. 
For $d=1$, $x \neq 0$, and $H \in (0,1)$, we show that $\xi(x) \in \mathcal{G}'$, the space of regular generalized functions. Elements of the space $\mathcal{G}'$ and elements from the negative Sobolev--Watanabe distribution spaces share the property that partial sums of their chaos decomposition are square integrable functions. More precisely, we show that $\xi(x) \in \mathcal{G}_{-s} \subset \mathcal{G}'$ for $x \neq 0$, $H \in (0,1)$, and all $s > 0$.
\end{abstract}

\begin{graphicalabstract}
\end{graphicalabstract}

\begin{highlights}
\item Using white noise analysis, we prove an existence result for stochastic currents of fractional Brownian motion in any dimension and for any Hurst parameter $H\in(0,1)$ as a Hida distribution.  

\item In the one-dimensional case, we establish a regularity result implying the square integrability of partial sums in the chaos decomposition for stochastic currents of fractional Brownian motion.
\end{highlights}

\begin{keyword}


Stochastic current \sep fractional Brownian motion \sep
fractional It{\^o} integral \sep white noise analysis, regular generalized functions.\\
\MSC[2020] 60H40 \sep 60G22 \sep 46F25 \sep 60G20
\end{keyword}

\end{frontmatter}



\section{Introduction}

The concept of currents has its origins in geometric measure theory. A typical $1$-current is given by
\[
\varphi\mapsto\int_{0}^{T}(\varphi(\gamma(t)),\gamma'(t))_{\mathbb{R}^{d}}\,\mathrm{d}t,\qquad0<T<\infty,\quad d\in\mathbb{N},
\]
where $\varphi:\mathbb{R}^{d}\rightarrow\mathbb{R}^{d}$ and $[0,T] \ni t \mapsto \gamma(t) \in \mathbb{R}^{d}$
is a rectifiable curve. The interested reader
may find definitions, results, and applications on the subject in the
books \cite{Federer1996, Mor16}.

To obtain its integral kernel, one can propose the ansatz
\[
\zeta(x):=\int_{0}^{T}\delta(x-\gamma(t))\gamma'(t)\,\mathrm{d}t,\quad x\in{\mathbb R}^d,
\]
where $\delta$ is the Dirac delta function, and try to give a mathematically rigorous meaning in an appropriate space of generalized functions.

The stochastic analog of the integral kernel $\zeta(x)$ rises if we substitute
the deterministic curve $\gamma$ by the sample path
of a stochastic process $X$ taking values in $\mathbb{R}^d$. Hence, we obtain the following kernel
\begin{equation}
\xi(x):=\int_{0}^{T}\delta(x-X(t))\,\mathrm{d}X(t), \quad x\in{\mathbb R}^d.\label{eq:stochastic-integral}
\end{equation}
The stochastic integral (\ref{eq:stochastic-integral}) must be
properly defined. More precisely, we choose $X$ to be a $d$-dimensional
fractional Brownian motion (fBm) $B_{H}$, with Hurst parameter $H\in(0,1)$. Therefore, the main objective of our study is
\begin{equation}
\xi(x):=\int_{0}^{T}\delta(x-B_{H}(t))\,\mathrm{d}B_{H}(t).\label{eq:Wick-type-Sintegral}
\end{equation}
The stochastic integral is
interpreted as a fractional It{\^o} integral developed in \cite{B03}. Other approaches, such as Malliavin calculus and stochastic integrals through regularization, were investigated in \cite{Flandoli2005,Flandoli2009,Flandoli2010} to study $\xi$. In \cite{Flandoli2005,Flandoli2009,Flandoli2010}, pathwise with probability one, $\xi$ was constructed as a random variable taking values in a negative Sobolev space. I.e., for a fixed path, $\xi$ is a generalized function and is therefore not pointwise defined in $x\in\mathbb{R}^{d}$. Moreover, in \cite{Flandoli2010}, for all $x \in \mathbb{R}$, the kernel $\xi(x)$ was constructed in a negative Sobolev--Watanabe distribution space using Malliavin calculus for $H \in [1/2,1)$.

In this work, we show that, if $x\in\mathbb{R}^{d}\backslash\{0\}$, $\xi(x)$ is a Hida distribution for any $H\in (0,1)$ and $d\ge 1$ while for $x=0\in \mathbb{R}^d$, $\xi(x)$ is a Hida distribution whenever $dH<1$; see Theorem~\ref{thm:current-Hdistr} and Corollary~\ref{rem:main-result}. For $d=1$, $x \neq 0$, and $H \in (0,1)$, we show that $\xi(x) \in \mathcal{G}'$, the space of regular generalized functions. The new idea to prove this type of regularity can be found in the key Lemma \ref{lem:newStranslem}. Elements of the space $\mathcal{G}'$ and elements from the negative Sobolev--Watanabe distribution spaces share the property that the partial sums of their chaos decomposition are square integrable functions. More precisely, we show that $\xi(x) \in \mathcal{G}_{-s} \subset \mathcal{G}'$ for $x \neq 0$, $H \in (0,1)$, and all $s > 0$; see Theorem \ref{thm:main-theorem}. The question of whether $\xi(x) \in \mathcal{G}_{0} = L^2(\mu)$ remains open. Should this be the case, the integral kernel of stochastic currents would even exist as a real-valued random variable.


The paper is organized as follows. In Section \ref{sec:WNA}, we recall
the background from the white noise analysis needed later. In Section
\ref{sec:Stochastic-Current}, we prove the existence result, and in Section \ref{sec:Regularity-SC-fBm}, we prove the regularity result. The appendix contains the proofs for certain technical lemmas.

\section{Gaussian White Noise Calculus}

\label{sec:WNA}In this section, we briefly recall the concepts and
results of white noise analysis used throughout this work. For a detailed
explanation; see, e.g., ~\cite{BK88b}, \cite{Hid75}, \cite{HKPS93},
\cite{HOUZ10}, \cite{Kuo96}, \cite{O94}.

\subsection{Hida Test and Generalized Functions\label{subsec:Hida}}

The starting point of the white noise analysis is the real Gelfand triple
\[
S_{d}\subset L_{d}^{2}\subset S'_{d},
\]
where $L_{d}^{2}:=L^{2}(\mathbb{R},\mathbb{R}^{d})$, $d\geq1$, is
the real Hilbert space of all vector-valued square-integrable functions
with respect to the Lebesgue measure on $\mathbb{R}$, $S_{d}$ and
$S'_{d}$ is the Schwartz space of vector-valued test functions
and tempered distributions, respectively. We denote the $L_{d}^{2}$-norm
by $\|\cdot\|_{2}$ and the dual pairing between $S'_{d}$ and $S_{d}$
by $\left\langle \cdot,\cdot\right\rangle $, which is a bilinear extension of the inner product on $L_{d}^{2}$, that is,
\[
\langle \varphi,f\rangle=\sum_{i=1}^{d}\int_{\mathbb{R}}\varphi_{i}(x)f_{i}(x)\,\mathrm{d}x,
\]
for all $f=(f_{1},...,f_{d})\in L_{d}^{2}$ and all $\varphi=(\varphi_{1},...,\varphi_{d})\in S_{d}$.
By the Minlos theorem, there is a unique probability measure $\mu$
on the $\sigma$-algebra $\mathcal{B}$ genera\-ted by the cylinder
sets on $S'_{d}$ with characteristic function $C$ given by 
\[
C(\varphi):=e^{-\frac{1}{2}\|\varphi\|_{2}^{2}}=\int_{S'_{d}}e^{i\left\langle \varphi,\omega\right\rangle }\,\mathrm{d}\mu(\omega),\quad\varphi\in S_{d}.
\]
In this way, we have defined the white noise measure space $(S'_{d},\mathcal{B},\mu)$.
Within this formalism, one can show that
\[
\left(\langle \mathbbm{1}_{[0,t)},w_{1}\rangle,...,\langle \mathbbm{1}_{[0,t)},w_{d}\rangle\right),\quad w=(w_{1},...,w_{d})\in S'_{d}, \, t \ge 0,
\]
has a continuous modification $B(t,w)$, which is a $d$-dimensional Brownian motion. Here, $\mathbbm{1}_{A}$ denotes the indicator function of the Borel
set $A \subset {\mathbb{R}}$ and $\langle \mathbbm{1}_{A},w_{i}\rangle$, $i = 1, \ldots, d$ is defined as an $L^2(\mu)$-limit.
For an arbitrary Hurst parameter $0<H<1$,
\[
\left(\langle \eta_{t},w_{1}\rangle,...,\langle \eta_{t},w_{d}\rangle\right),\quad w=(w_{1},...,w_{d})\in S'_{d},\;\eta_{t}:=M_-^H\mathbbm{1}_{[0,t)},
\]
has a continuous modification $B_{H}(t,w)$ that is a $d$-dimensional fBm. On the one hand, for a real-valued function $f$ and $\frac{1}{2}<H<1$, the operator $M_-^H$ is defined by
\begin{equation}
(M_-^{H}f)(x):=K_H\cdot(I_-^Hf)(x):=\frac{K_{H}}{\Gamma\left(H-\frac{1}{2}\right)}\int_{0}^{\infty}f(x+t)t^{H-\frac{3}{2}}\,\mathrm{d}t,\label{08eq1}
\end{equation}
provided that the integral exists for almost all $x\in\mathbb{R}$ and the normalization constant is given by
\[
K_{H}:=\Gamma\left(H+\frac{1}{2}\right)\left(\frac{1}{2H}+\int_{0}^{\infty}\left((1+s)^{H-\frac{1}{2}}-s^{H-\frac{1}{2}}\right)\mathrm{d}s\right)^{-\frac{1}{2}}.
\]
On the other hand, for a real-valued function $f$ and $0<H<\frac{1}{2}$, the operator $M_-^H$ has the form
\begin{equation}
(M_-^{H}f)(x):=K_H\cdot(D_-^Hf)(x):=\frac{(\frac{1}{2}-H)K_{H}}{\Gamma\left(H+\frac{1}{2}\right)}\lim_{\varepsilon\to0^{+}}\int_{\varepsilon}^{\infty}\frac{f(x)-f(x+y)}{y^{\frac{3}{2}-H}}\,\mathrm{d}y,\label{08eq2}
\end{equation}
if the limit exists for almost all $x\in\mathbb{R}$. For $H=1/2$, $M_{-}^{1/2}$ is the identity operator. For more
details, see, e.g., \cite{B03}, \cite{PTaqqu00}, and the references therein. 

To introduce the corresponding fractional noise $W_H$, we first need to define the operator $M_+^H$, which is dual to the operator $M_-^H$. Therefore, for $\frac{1}{2}<H<1$ we define
\[
(M_+^Hf)(x):=K_H\cdot(I_+^Hf)(x):=K_H\cdot(f*g_H)(x),
\]
where $g_H(t):=\frac{1}{\Gamma(H)}t^H$, $t>0$, whenever the convolution integral exists for almost all $x\in\mathbb{R}$. For $0<H<\frac{1}{2}$ the operator $M_+^H$ is defined by
\[
(M_+^Hf)(x):=K_H\cdot(D_+^Hf)(x):=\frac{(\frac{1}{2}-H)K_{H}}{\Gamma\left(H+\frac{1}{2}\right)}\lim_{\varepsilon\to0^{+}}\int_{\varepsilon}^{\infty}\frac{f(x)-f(x-y)}{y^{\frac{3}{2}-H}}\,\mathrm{d}y,
\]
if the limit exists for almost all $x\in\mathbb{R}$.

There are several examples of functions $f$ for which $M_{\pm}^Hf$ exists for any $H\in(0,1)$. For example, $f=\mathbbm{1}_{[0,t)}$,
$t\geq0$, or $f\in S_{1}(\mathbb{R})$. For functions $f_{1}$, $f_{2}$
being either of these two types, it is easy to prove the following
equality
\[
\int_{\mathbb{R}}f_{1}(x)(M_{-}^Hf_{2})(x)\,\mathrm{d}x=\int_{\mathbb{R}}(M_{+}^{H}f_{1})(x)f_{2}(x)\,\mathrm{d}x,
\]
showing that $M_{-}^H$ and $M_{+}^{H}$ are dual operators, cf.~Eq.~(12) in \cite{B03}.

The corresponding $d$-dimensional fractional noise $W_{H}(t)$ in the sense of Hida distributions
is given by 
\begin{equation} \label{eq:fractional-WN}
W_{H}(t) := (W_{H,1}(t),\ldots,W_{H,d}(t)):=(\langle M_{-}^{H}\delta_t,P_{1}\rangle,\ldots,\langle  M_{-}^{H}\delta_t,P_{d}\rangle),
\end{equation}
where $P_i: S'_{d} \to S'_{1}, i = 1, \ldots, d$, denotes the projection on the $i$-th compo\-nent, see Definition 2.18 in \cite{B03} for $d=1$. The generalized function $M_{-}^{H}\delta_t\in S'_1$ is defined for any test function $\varphi\in S_1$ by 
\begin{equation}\label{eq:M-minus-delta}
\langle\varphi, M_{-}^{H}\delta_t\rangle:=(M_+^H\varphi)(t).
\end{equation}
 For $H=1/2$, the operator
$M_{+}^{1/2}$ is defined as the identity, and $W_{1/2}(t)=(\langle \delta_t,P_{1}\rangle,\ldots,\langle \delta_t,P_{d}\rangle)$
coincides with the vector-valued white noise.

Let us now consider the complex Hilbert space $L^{2}(\mu):=L^{2}(S'_{d},\mathcal{B},\mu;\mathbb{C})$.
This space is canonically isomorphic to the symmetric Fock space of
symme\-tric square-integrable functions, 
\[
L^{2}(\mu)\simeq\Big(\bigoplus_{k=0}^{\infty}\mathrm{Sym}\,L^{2}(\mathbb{R}^{k},k!\mathrm{d}^{k}x)\Big)^{\otimes d},
\]
leading to the chaos expansion of the elements in $L^{2}(\mu)$, 
\begin{equation}
F(w_{1},...,w_{d})=\sum_{(n_{1},...,n_{d})\in\mathbb{N}_{0}^{d}}\langle F_{(n_{1},...,n_{d})},:w_{1}^{\otimes n_{1}}:\otimes\cdots\otimes:w_{d}^{\otimes n_{d}}:\rangle,\label{08eq5}
\end{equation}
with kernel functions $F_{(n_{1},...,n_{d})}\in (L^2_{d,\mathbb{C}})^{\hat{\otimes}n}$ and $w=(w_1,\dots,w_d)\in S'_d$.
For simplicity, we use the notation
\begin{eqnarray*}
\bm{n}&=&(n_{1},\cdots,n_{d})\in\mathbb{N}_{0}^{d},\quad n=\sum_{i=1}^{d}n_{i},\quad\bm{n}!=\prod_{i=1}^{d}n_{i}!,\\ 
:w^{\otimes\bm{n}}:&:=&:w_{1}^{\otimes n_{1}}:\otimes\cdots\otimes:w_{d}^{\otimes n_{d}}:,
\end{eqnarray*}
that reduces the chaos expansion in (\ref{08eq5}) to
\[
F(w)=\sum_{\bm{n}\in\mathbb{N}_{0}^{d}}\langle F_{\bm{n}},:w^{\otimes\bm{n}}:\rangle,\quad w\in S'_{d}.
\]

To proceed further, we have to consider a Gelfand triple around the
space $L^{2}(\mu)$. We use the space $(S_{d})'$ of Hida
distributions and the corresponding
Gelfand triple 
\begin{equation}\label{eq:Gelfand-triple}
(S_{d})\subset L^{2}(\mu)\subset(S_{d})'.
\end{equation}
Here $(S_{d})$ is the space of the white noise test functions such that
its dual space (with respect to $L^{2}(\mu)$) is the space $(S_{d})'$.
Instead of reproducing the explicit construction of $\left(S_d\right)'$
(see e.g., \cite{HKPS93}), 
we characterize this space by its $S$-transform in Theorem~\ref{thm:charact-thm-Hida}. We recall that given a $\varphi\in S_{d}$, the Wick exponential is defined by
\[
:\exp(\langle \varphi, w\rangle):\::=\sum_{\bm{n}\in\mathbb{N}_{0}^{d}}\frac{1}{\bm{n}!}\langle \varphi^{\otimes\bm{n}},:w^{\otimes\bm{n}}:\rangle=C(\varphi)e^{\langle \varphi, w\rangle},
\]
where $\varphi^{\otimes\bm{n}}:=\varphi_1^{\otimes n_1}\otimes\dots\otimes\varphi_d^{\otimes n_d}$ and the $S$-transform of  $\Phi\in(S_d)'$ is defined by
\begin{equation}
S\Phi(\varphi):=\left\langle \!\left\langle :\exp(\left\langle\varphi, \cdot\right\rangle ):,\Phi\right\rangle \!\right\rangle ,\quad \varphi\in S_{d}.\label{08eq6}
\end{equation}
Here $\left\langle \!\left\langle \cdot,\cdot\right\rangle \!\right\rangle $
denotes the dual pairing between $\left(S_d\right)$ and $\left(S_d\right)'$,
which is a bilinear extension of $\left( \!\left( \cdot,\overline{\cdot}\right) \!\right)$, where $\left( \!\left( \cdot,\cdot\right) \!\right)$ is the sesquilinear inner product on $L^{2}(\mu)$. We observe that the multilinear expansion
of (\ref{08eq6}), 
\[
S\Phi(\varphi):=\sum_{\bm{n}\in\mathbb{N}_{0}^{d}}\langle\varphi^{\otimes\bm{n}}, \Phi_{\bm{n}}\rangle,
\]
extends the chaos expansion to $\Phi\in\left(S_d\right)'$ with distribution
valued kernels $\Phi_{\bm{n}} \in (S'_d)^{\otimes\bm{n}}$ such that
\[
\left\langle \!\left\langle \varphi,\Phi\right\rangle \!\right\rangle =\sum_{\bm{n}\in\mathbb{N}_{0}^{d}}\bm{n}!\langle\varphi_{\bm{n}},\Phi_{\bm{n}}\rangle,
\]
for every test function $\varphi\in(S_{d})$ with kernel
functions $\varphi_{\bm{n}} \in (S_d)^{\otimes\bm{n}}$. This allows us to represent $\Phi$ by its generalized chaos expansion
\[
\Phi=\sum_{\bm{n}\in\mathbb{N}_{0}^{d}} I_{\bm{n}}(\Phi_{\bm{n}}),\quad \Phi_{\bm{n}}\in (S'_d)^{\otimes\bm{n}},
\]
where
\[
\left\langle \!\left\langle \varphi,I_{\bm{n}}(\Phi_{\bm{n}})\right\rangle \!\right\rangle
:= \bm{n}!\langle\varphi_{\bm{n}},\Phi_{\bm{n}}\rangle,\quad \varphi\in(S_{d}).
\]
\begin{example}\label{exa:fwn}
Let $d=1$ and $W_{H}(t)$ be the fractional noise introduced
in \eqref{eq:fractional-WN}. Then its $S$-transform is given by
(cf.~\cite{B03})
\[
SW_{H}(t)(\varphi)=(M_{+}^{H}\varphi)(t),\quad \varphi\in S_{d}.
\]
\end{example}

To characterize the space $(S_d)'$ through its $S$-transform, we first need the following definition.
\begin{defn}[$U$-functional]
\label{def:U-functional}A function $F:S_{d}\rightarrow\mathbb{C}$
is called a $U$-functional whenever
\begin{enumerate}
\item for every $\varphi_{1},\varphi_{2}\in S_{d}$ the mapping $\mathbb{R}\ni\lambda\longmapsto F(\lambda\varphi_{1}+\varphi_{2})$
has an entire extension to $\lambda\in\mathbb{C}$,
\item there are constants $K_{1},K_{2} < \infty$ such that 
\[
\left|F(z\varphi)\right|\leq K_{1}\exp\big(K_{2}|z|^{2}\|\varphi\|^{2}\big),\quad z\in\mathbb{C},\varphi\in S_{d}
\]
for some continuous norm $\|\cdot\|$ on $S_{d}$.
\end{enumerate}
\end{defn}

We are now ready to state the characterization theorem mentioned above.
\begin{thm}[cf.~\cite{PS91}, \cite{KLPSW96}]
\label{thm:charact-thm-Hida}The $S$-transform defines a bijection
between the space $(S_d)'$ and the space of the $U$-functionals.
\end{thm}

As a consequence of Theorem \ref{thm:charact-thm-Hida}, one may derive
the following statement, which concerns the Bochner integration
of a family of the same type of distributions. For more details and
proofs; see, e.g.,~\cite{PS91}, \cite{HKPS93}, \cite{KLPSW96}.
\begin{cor}
\label{cor:hida_integral}Let $(\Omega,\mathcal{F},m)$ be a measure
space and $\lambda\mapsto\Phi_{\lambda}$ be a mapping from $\Omega$
to $(S_{d})'$. We assume that the $S$-transform of $\Phi_{\lambda}$
fulfills the following two properties:
\begin{enumerate}
\item The mapping $\lambda\mapsto S\Phi_{\lambda}(\varphi)$ is measurable
for every $\varphi\in S_{d}$.
\item The function $S\Phi_{\lambda}$ obeys the estimate
\[
|S\Phi_{\lambda}(z\varphi)|\leq C_{1}(\lambda)e^{C_{2}(\lambda)|z|^{2}\|\varphi\|^{2}},\quad z\in\mathbb{C},\varphi\in S_{d},
\]
for some continuous norm $\Vert\cdot\Vert$ on $S_{d}$ and
$C_{1}\in L^{1}(\Omega,m)$, $C_{2}\in L^{\infty}(\Omega,m)$.
\end{enumerate}
Then 
\[
\int_{\Omega}\Phi_{\lambda}\,\mathrm{d}m(\lambda)\in(S_{d})',
\]
and 
\[
S\left(\int_{\Omega}\Phi_{\lambda}\,\mathrm{d}m(\lambda)\right)(\varphi)=\int_{\Omega}S\Phi_{\lambda}(\varphi)\,\mathrm{d}m(\lambda),\qquad\varphi\in S_{d}.
\]

\end{cor}

\begin{example}[Donsker's delta function]
\label{exa:S-transform-delta}As a typical example of a Hida distribution,
we have the Donsker delta function needed later. More precisely,
the following Bochner integral is a well-defined element in $(S_{d})'$:
\[
\Phi_{x,H}:=\delta(x-B_{H}(t))=\frac{1}{(2\pi)^{d}}\int_{\mathbb{R}^{d}}e^{i(\lambda,x-B_{H}(t))_{\mathbb{R}^{d}}}\,\mathrm{d}\lambda,\quad x\in\mathbb{R}^{d}.
\]
In fact, the $S$-transform of $\Phi_{x,H}$ for any $z\in\mathbb{C}$ and
$\varphi\in S_{d}$ is given by
\begin{equation}
S\Phi_{x,H}(z\varphi)=\frac{1}{(2\pi t^{2H})^{d/2}}\exp\left(-\frac{1}{2t^{2H}}\sum_{j=1}^{d}(x_{j}-\langle z\varphi_{j},\eta_{t}\rangle)^{2}\right).\label{eq:S-Donsker}
\end{equation}
\end{example}

It is well known that the Wick product is a well-defined operation
in Gaussian analysis; see, for example, \cite{KLS96}, \cite{HOUZ10}, and \cite{KSWY95}.
\begin{defn}
\label{def:Wick_product}For any $\Phi,\Psi\in(S_{d})'$, the Wick
product $\Phi\Diamond\Psi$ is defined by 
\begin{equation}
S(\Phi\Diamond\Psi)=S\Phi\cdot S\Psi.\label{eq:Wick_product}
\end{equation}
Since the space of $U$-functionals is an algebra, by
Theorem~\ref{thm:charact-thm-Hida} there exists a unique element
$\Phi\Diamond\Psi\in(S_{d})'$ such that (\ref{eq:Wick_product})
holds.
\end{defn}

\subsection{Characterization of Regular Generalized Functions\label{sec:Characterisation-Regular}}
In this subsection, we introduce the space of regular generalized functions, which is used in Section~\ref{sec:Regularity-SC-fBm}.
We restrict the construction to $d=1$ and adapt the notation from
Subsection~\ref{subsec:Hida}.

Consider the Wiener-It\^o chaos decomposition of $F\in L^{2}(\mu)$
\[
F=\sum_{n=0}^{\infty}\left\langle f^{(n)},:\cdot^{\otimes n}:\right\rangle ,\quad f^{(n)}\in(L_{1,\mathbb{C}}^{2})^{\hat{\otimes}n},
\]
with 
\begin{equation}
\|F\|_{L^{2}(\mu)}^{2}=\sum_{n=0}^{\infty}n!|f^{(n)}|_{(L_{1,\mathbb{C}}^{2})^{\otimes n}}^{2},\label{eq:L2norm}
\end{equation}
where $:\omega^{\otimes n}:\in(S'_{1,\mathbb{C}})^{\hat{\otimes}n}$
denotes the Wick power of $n$-th order of $\omega\in S'_{1,\mathbb{C}}$ and $\langle\cdot,\cdot\rangle$
denotes the bilinear dual pairing on $S_{\mathbb{C}}(\mathbb{R})^{\otimes n}\times S'_{\mathbb{C}}(\mathbb{R})^{\otimes n}$.
We remark that $\langle\cdot,:\cdot^{\otimes n}:\rangle$ extends
in the first variable to $(L_{1,\mathbb{C}}^{2})^{\hat{\otimes}n}$
in the sense of an $L^{2}(\mu)$-limit.

Now, let $s\in [0,\infty)$ be given. If we require 
\[
\|F\|_{\mathcal{G}_{s}}^{2}:=\sum_{n=0}^{\infty}n!2^{ns}|f^{(n)}|_{(L_{1,\mathbb{C}}^{2})^{\otimes n}}^{2}
\]
to be finite, in addition to \eqref{eq:L2norm}, we obtain 
\[
\mathcal{G}_{s}:=\left\{ F=\sum_{n=0}^{\infty}\left\langle f^{(n)},:\cdot^{\otimes n}:\right\rangle \in L^{2}(\mu)\mid f^{(n)}\in(L_{1,\mathbb{C}}^{2})^{\hat{\otimes}n},\|F\|_{\mathcal{G}_{s}}<\infty\right\} .
\]

Thus, $\left(\mathcal{G}_{s}\right)_{s\in[0,\infty)}$ is a nested
family of Hilbert subspaces of $L^{2}(\mu)$ with $\mathcal{G}_{0}=L^{2}(\mu)$.
We write $\mathcal{G}:=\bigcap_{s\in[0,\infty)} \mathcal{G}_{s}$
for the projective limit of this family. This yields the topological
dual space $\mathcal{G}'=\bigcup_{s\in [0,\infty)}\mathcal{G}_{-s}$,
where $\mathcal{G}_{-s}:=\mathcal{G}'_{s}$. The elements of $\mathcal{G}'$ are referred to as regular generalized functions because the partial sums within their chaos decomposition are square integrable functions.

Note that the construction of $\mathcal{G}$ is almost identical to
 that of $(S_1)$. The difference lies, in addition to the slightly weaker summability assumption, in the fact that we do not require the kernels $f^{(n)}$ to possess additional regularity. Hence, we have the
following refinement of the Gelfand triple \eqref{eq:Gelfand-triple}
with $d=1$
\begin{equation}
(S_{1})\subset\mathcal{G}\subset\mathcal{G}_{s}\subset L^{2}(\mu)\subset\mathcal{G}_{-s}\subset\mathcal{G}'\subset (S_{1})',\qquad s\in[0,\infty).\label{gelfand3}
\end{equation}

\begin{example}
\label{exa:gelfand3example}Let $t\geq0$ be fixed and $H=\frac{1}{2}$.
Then the Brownian motion $B_{\frac{1}{2}}(t)=\langle\mathbbm{1}_{[0,t)},\cdot\rangle$
and white noise $W_{\frac{1}{2}}(t)=\langle\delta_{t},\cdot\rangle$
described above are such that $B_{\frac{1}{2}}(t)\in\mathcal{G}\setminus(S_{1})$
and $W_{\frac{1}{2}}(t)\in (S_{1})'\setminus\mathcal{G}'$.
\end{example}

In the following, we provide the characterization of regular generalized
functions, see \cite{Grothaus2021} for details. More precisely,
given a $\Phi\in (S_{1})'$, we have a necessary and sufficient condition
to check whether $\Phi\in\mathcal{G}_{s}$ for some $s\in\mathbb{Z}$.
Before we state this characterization theorem (cf.~Theorem~\ref{thm:charact-regular} below), we need some preparations.
Recall the Gaussian measure $\mu_{\sigma^2}$ on $(S_{1}',\mathcal{B})$ with covariance $\sigma^2>0$ is given by its characteristic functional
\begin{equation}\label{charfct}
C_{\sigma^{2}}(\varphi):=\exp\left(-\frac{\sigma^{2}}{2}\|\varphi\|_{2}^{2}\right),\quad\varphi\in S_{1}.
\end{equation}
\begin{defn}
\label{def:complexwn}We identify $S'_{1,\mathbb{C}}$ with $S_{1}'\times S_{1}'$
and define the product measure $\nu:=\mu_{\frac{1}{2}}\otimes\mu_{\frac{1}{2}}$
on $(S'_{1,\mathbb{C}},\mathcal{B}\otimes\mathcal{B})$.
\end{defn}

To characterize elements $\Phi \in \mathcal{G}_{s}$, it suffices to consider their $S$-transform $S\Phi$, a function defined on $S_1$. Since it is also a $U$-functional by Theorem~\ref{thm:charact-thm-Hida},
there is a unique extension of $S\Phi$ to an entire function on $S_{1,\mathbb{C}}$,
see Lemma~10 in \cite{KLPSW96}.

Starting with an element $z\in S'_{1,\mathbb{C}}$, we want to project
it into $S_{1,\mathbb{C}}$ to which we may apply $S\Phi$.
\begin{defn}
Let $n\in\mathbb{N}$ be given and $\left(\varphi_{i}\right)_{i=1}^n\subset S_{1}$ be orthonormal in $L_{1}^{2}$.
We call 
\[
P:S'_{1,\mathbb{C}}\longrightarrow S_{1,\mathbb{C}},\;z\mapsto Pz:=\sum_{i=1}^{n}\langle\varphi_{i},z\rangle\ \varphi_{i},
\]
an \emph{orthogonal projection} from $S'_{1,\mathbb{C}}$ into $S_{1,\mathbb{C}}$.
Here, $\langle\cdot,\cdot\rangle$ denotes the dual pairing on $S_{1,\mathbb{C}}\times S'_{1,\mathbb{C}}$
which is a bilinear extension of $\left(\cdot,\overline{\cdot}\right)_{L_{1,\mathbb{C}}^{2}}$.
The set of all orthogonal projections from $S'_{1,\mathbb{C}}$ into
$S_{1,\mathbb{C}}$ is denoted by $\mathbb{P}$.
\end{defn}

Like orthogonal projections in the classical sense, any $P\in\mathbb{P}$
has the following property 
\begin{equation}
\langle Pz,z^{\prime}\rangle=\left\langle \sum_{i=1}^{n}\langle\varphi_{i},z\rangle\varphi_{i},z^{\prime}\right\rangle =\sum_{i=1}^{n}\langle\varphi_{i},z\rangle\langle\varphi_{i},z^{\prime}\rangle=\langle Pz^{\prime},z\rangle,\label{eq:projprop}
\end{equation}
for every $z,z^{\prime}\in S'_{1,\mathbb{C}}$. We are now able to state
the characterization theorem and refer to Theorem~2.11 in \cite{Grothaus2021} for the proof and more details.
\begin{thm}
\label{thm:charact-regular} Let $\Phi\in(S_1)^{\prime}$ be given.
Then for $s\in\mathbb{R}$ it holds
\[
\Phi\in\mathcal{G}_{s}\Longleftrightarrow\sup\limits_{P\in\mathbb{P}}\int_{S'_{1,\mathbb{C}}}\left|S\Phi\left(2^{\frac{s}{2}}Pz\right)\right|^{2}\mathrm{d}\nu(z)<\infty.
\]
In particular, the spaces $\mathcal{G}$ and $\mathcal{G}'$ are characterized
by 
\end{thm}

\begin{enumerate}
\item[(i)] $\Phi\in\mathcal{G}\Longleftrightarrow\forall\lambda>0,\;\sup_{P\in\mathbb{P}}\int_{S'_{1,\mathbb{C}}}\left|S\Phi\left(\lambda Pz\right)\right|^{2}\mathrm{d}\nu(z)<\infty,$ 
\item[(ii)] $\Phi\in\mathcal{G}'\Longleftrightarrow\exists\varepsilon>0,\;\sup_{P\in\mathbb{P}}\int_{S'_{1,\mathbb{C}}}\left|S\Phi\left(\varepsilon Pz\right)\right|^{2}\mathrm{d}\nu(z)<\infty.$ 
\end{enumerate}

\section{Stochastic Current of Fractional Brownian Motion}
\label{sec:Stochastic-Current}
As motivated in the introduction, using white noise analysis, we investi\-gate for $x \in {\mathbb R}^d$ the following (generalized) function  
\begin{align*}
\xi(x) & :=\int_{0}^{T}\delta(x-B_{H}(t))\,\mathrm{d}B_{H}(t)\\
 & :=\left(\int_{0}^{T}\delta(x-B_H(t))\Diamond W_{H,1}(t)\,\mathrm{d}t,\ldots,\int_{0}^{T}\delta(x-B_H)(t))\Diamond W_{H,d}(t)\,\mathrm{d}t\right)\\
 & =:\big(\xi_{1}(x),\ldots,\xi_{d}(x)\big),
\end{align*}
where $W_H:=(W_{H,1},\dots,W_{H,d})$ is the vector-valued fractional noise defined in ~\eqref{eq:fractional-WN}. The above stochastic integral was 
introduced in \cite[Equation~(26)]{B03} and is called the fractional It\^o integral. If $H=1/2$ and the integrand is an adapted, square integrable function, then this stochastic integral coincides with the classical It\^o integral.

From this point onward, $C$ will denote a positive finite constant whose value may change from line to line.

\begin{thm}
\label{thm:current-Hdistr}For $x\in\mathbb{R}^{d}\backslash\{0\}$, $0<T<\infty$, $H\in (0,1)$, $d\geq1$, and for each $i=1,\ldots,d$, the Bochner integral 
\begin{equation}
\xi_{i}(x)=\int_{0}^{T}\delta(x-B_{H}(t))\Diamond W_{H,i}(t)\,\mathrm{d}t\label{eq:current-bochner}
\end{equation}
is a Hida distribution, and its $S$-transform is given, for any $\varphi\in S_{d}$,
by
\begin{equation}\label{eq:S-transf-current-fBm}
S\left(\xi_{i}(x)\right)(\varphi)=\frac{1}{(2\pi)^{d/2}}\int_{0}^{T}\frac{1}{t^{Hd}}\mathrm{e}^{-\sum_{j=1}^{d}\frac{(x_{j}-\langle\varphi_{j},\eta_{t}\rangle)^{2}}{2t^{2H}}}(M_{+}^{H}\varphi_{i})(t)\,\mathrm{d}t.
\end{equation}
\end{thm}

\begin{proof}
First, we compute the $S$-transform of the integrand $\Phi_t$, $t\in(0,T]$, in (\ref{eq:current-bochner}), that is,
\[
\Phi_t:=\delta(x-B_{H}(t))\Diamond W_{H,i}(t).
\]
It follows from Definition~\ref{def:Wick_product}, Examples~\ref{exa:fwn} and \ref{exa:S-transform-delta} that, for any $\varphi\in S_{d}$, we obtain
\begin{eqnarray}
S\Phi_t(\varphi)  &=&S\big(\delta(x-B_{H}(t))\big)(\varphi)S\big(W_{H,i}(t)\big)(\varphi)\nonumber\\
  &=&\frac{1}{(2\pi t^{2H})^{d/2}}\exp\left(-\frac{1}{2t^{2H}}\sum_{j=1}^{d}(x_{j}-\langle\varphi_{j},\eta_{t}\rangle)^{2}\right)(M_{+}^{H}\varphi_{i})(t)\label{eq:ST-Integ-fBm}.
\end{eqnarray}
It is clear that $(0,T] \ni t \mapsto S\Phi_t(\varphi) \in {\mathbb C}$ is Borel measurable for every $\varphi\in S_{d}$.
Furthermore, for any $z\in\mathbb{C}$ and all $\varphi\in S_{d}$, we estimate 
$|S\Phi_t(z\varphi)|$ as follows
\begin{align*}
& |S\Phi_t(z\varphi)| \\
&\le \frac{1}{(2\pi t^{2H})^{d/2}} \exp\left(-\frac{1}{2t^{2H}}|x|_{\mathbb{R}^d}^2\right)\exp\left(\frac{1}{2t^{2H}}\sum_{j=1}^d|x_j||z||\langle\varphi_{j},\eta_{t}\rangle|\right)\\
&\times \exp\left(\frac{1}{2t^{2H}}|z|^2\sum_{j=1}^d|\langle\varphi_{j},\eta_{t}\rangle|^2\right)|z|\left| M_+^H\varphi_i(t)\right|\\
&\le \frac{1}{(2\pi t^{2H})^{d/2}} \exp\left(-\frac{1}{2t^{2H}}|x|_{\mathbb{R}^d}^2\right)\exp\left(\frac{1}{2t^{2H}}\sum_{j=1}^d\left( \frac{1}{4}|x_j|^2+|z|^2|\langle\varphi_{j},\eta_{t}\rangle|^2\right)\right)\\
&\times \exp\left(\frac{1}{2t^{2H}}|z|^2t^{2H}\|\varphi\|_{2}^2\right)|z|\| M_+^H\varphi\|_{\infty}\\
 & \le\frac{1}{(2\pi t^{2H})^{d/2}}\exp\left(-\frac{1}{4t^{2H}}|x|_{\mathbb{R}^d}^2\right) \exp\left( C|z|^2\|\varphi\|_2^2\right) \exp\left( |z|^2\| M_+^H\varphi\|_{\infty}^2\right)\\
 & \le\frac{1}{(2\pi t^{2H})^{d/2}}\exp\left(-\frac{1}{4t^{2H}}|x|_{\mathbb{R}^d}^2\right)\exp\left( C|z|^2\|\varphi\|^2\right)
\end{align*}
for some continuous norm $\|\cdot\|$ on $S_d$. The third inequality is a consequence of
\[
|\langle\varphi,\eta_{t}\rangle|^{2}\le\|\varphi\|_{2}^{2}\|M_{-}^{H}\mathbbm{1}_{[0,t)}\|_2^2=t^{2H}\|\varphi\|_{2}^{2}.
\] 
 The last inequality follows from $\|M_+^H\varphi\|_\infty^2\le\|\varphi\|^2$, which is due to Theo\-rem~2.3 in \cite{B03}. By using a gamma function argument, the function $(0,T] \ni t \mapsto \frac{1}{\left(2\pi t^{2H}\right)^{d/2}}\exp\left(-\frac{1}{4t^{2H}}|x|^{2}_{\mathbb{R}^d}\right)$ is integrable with respect to the Lebesgue measure on $[0,T]$.

As the second factor $\exp\big(C|z|^2\|\varphi\|_2^{2}\big)$
 is independent of $t\in[0,T]$, 
this shows that the conditions of Corollary~\ref{cor:hida_integral} are satisfied and
\[
\int_{0}^{T}\delta(x-B_{H}(t))\Diamond W_{H,i}(t)\,\mathrm{d}t\in(S_{d})'.\qedhere
\]
\end{proof}

Analyzing the proof of Theorem~\ref{thm:current-Hdistr}, we see that it is also possible to include $x=0\in\mathbb{R}^d$.

\begin{cor}\label{rem:main-result}
\begin{enumerate}
    \item For $d=1$ and for all $H\in(0,1)$ we have $\xi(0)\in (S_1)'$.
    \item For $d\geq2$ and $H\in(0,1/d)$ we have $\xi(0)\in(S_d)'$. 
\end{enumerate}
\end{cor}

\begin{rem}
To cover the case $H\in[1/d,1)$, we may truncate $\xi(0)$ and use techniques similar to those in the analysis of local times for fBm; see, for example, \cite{drumond-oliveira-silva08, OSS09} and other sources cited therein. 
\end{rem}

\section{Regularity of Stochastic Currents of Fractional Brownian Motion\label{sec:Regularity-SC-fBm}}
Consider the $S$-transform of the stochastic current of 1-dimensional fBm with Hurst parameter $H\in(0,1)$:
\begin{equation}
S\left(\xi(x)\right)(\varphi)=\frac{1}{(2\pi)^{1/2}}\int_{0}^{T}\frac{1}{t^{H}}\mathrm{e}^{-\frac{(x-\langle\varphi,\eta_{t}\rangle)^{2}}{2t^{2H}}}(M_{+}^{H}\varphi)(t)\,\mathrm{d}t,\qquad\varphi\in S_{1}.\label{eq:1d-S-transf-current-fBm}
\end{equation}
To show the integrability condition of Theorem \ref{thm:charact-regular}
of the $S$-transform \eqref{eq:1d-S-transf-current-fBm}, more suitable representation of $S(\xi(x))$ is needed. Point evaluations are too singular. The test function should appear only in integrated form. The new idea for obtaining such a representation can be found in the following key lemma:

\begin{lem}\label{lem:newStranslem}
For every $x\in\mathbb{R}$, $H\in(0,1)$, the $S$-transform
of the stochastic current $\xi(x)$ of fBm has a representation
given by
\begin{equation}
\begin{aligned}
S\left(\xi(x)\right)(\varphi) & =\frac{H}{\sqrt{2\pi}}\int_{0}^{T}\int_{0}^{1}\frac{1}{t^{H+1}}\mathrm{e}^{-\frac{(x-\tau\langle\varphi,\eta_{t}\rangle)^{2}}{2t^{2H}}}\langle\varphi,\eta_{t}\rangle\,\mathrm{d}\tau\,\mathrm{d}t\\
 & +\frac{1}{\sqrt{2\pi}T^{H}}\int_{0}^{1}\mathrm{e}^{-\frac{(x-\tau\langle\varphi,\eta_{T}\rangle)^{2}}{2T^{2H}}}\langle\varphi,\eta_{T}\rangle\,\mathrm{d}\tau\\
 & -\frac{H}{\sqrt{2\pi}}\int_{0}^{T}\int_{0}^{1}\frac{(x-\tau\langle\varphi,\eta_{t}\rangle)^{2}}{t^{3H+1}}\mathrm{e}^{-\frac{(x-\tau\langle\varphi,\eta_{t}\rangle)^{2}}{2t^{2H}}}\langle\varphi,\eta_{t}\rangle\,\mathrm{d}\tau\,\mathrm{d}t,
 \end{aligned}
 \label{eq:newStrans}
\end{equation}
for any $\varphi\in S_{1}$.
\end{lem}
\begin{proof} \textit{Step 1}:
For any test function $\varphi\in S_{1}$ and every $x\in\mathbb{R}$ we define the maps
\begin{align*}
f & :(0,\infty)\times\mathbb{R}\longrightarrow\mathbb{R},\;(t,u)\mapsto f(t,u):=\int_{0}^{u}\exp\left(-\frac{(x-v)^{2}}{2t^{2H}}\right)\mathrm{d}v,\\
F: & (0,\infty)\longrightarrow\mathbb{R},\;t\mapsto F(t):=f(t,\psi(t)),\quad\psi(t):=\langle\varphi,\eta_{t}\rangle.
\end{align*}
By Corollary 2.8 in \cite{B03}, $\psi:(0,\infty)\longrightarrow\mathbb{R}$ is differentiable, and its derivative $t\mapsto \psi'(t)=(M_+^H\varphi)(t)$ is a continuous function. It follows that $F$ is differentiable, and its derivative is given
by
\begin{align}\label{eq:Fprime}
&F'(t)=\frac{\partial f}{\partial t}(t,\psi(t))+\frac{\partial f}{\partial u}(t,\psi(t))\frac{\partial\psi}{\partial t}(t)  \\
& =H\int_{0}^{\psi(t)}\frac{(x-v)^{2}}{t^{2H+1}}\exp\left(-\frac{(x-v)^{2}}{2t^{2H}}\right)\mathrm{d}v +\exp\left(-\frac{(x-\psi(t))^{2}}{2t^{2H}}\right)(M_{+}^{H}\varphi)(t).\nonumber
\end{align}
For every $t>0$, the change of variable $\psi(t)\tau=v$ in the integral
above yields
\begin{multline}\label{eq:integrand-3}
\int_{0}^{\psi(t)}\frac{(x-v)^{2}}{t^{2H+1}}\exp\left(-\frac{(x-v)^{2}}{2t^{2H}}\right)\mathrm{d}v \\
=\int_{0}^{1}\frac{(x-\tau\psi(t))^{2}}{t^{2H+1}}\exp\left(-\frac{(x-\tau\psi(t))^{2}}{2t^{2H}}\right)\psi(t)\,\mathrm{d}\tau.
\end{multline}
It follows from \eqref{eq:Fprime} and \eqref{eq:integrand-3} that
\begin{equation}\label{eq:Fprime-divided}
\begin{aligned}
\frac{F'(t)}{\sqrt{2\pi}t^{H}}  =\frac{H}{\sqrt{2\pi}}\frac{1}{t^{H}}\int_{0}^{1}\frac{(x-\tau\psi(t))^{2}}{t^{2H+1}}\exp\left(-\frac{(x-\tau\psi(t))^{2}}{2t^{2H}}\right)\psi(t)\mathrm{d}\tau\\
  +\frac{1}{\sqrt{2\pi}t^{H}}\exp\left(-\frac{(x-\psi(t))^{2}}{2t^{2H}}\right)(M_{+}^{H}\varphi)(t).
\end{aligned}
\end{equation}
\textit{Step 2}: The function
\begin{equation}\label{eq:step2}
(0,T]\ni t\mapsto \int_{0}^{1}\frac{(x-\tau\psi(t))^{2}}{t^{2H+1}}\exp\left(-\frac{(x-\tau\psi(t))^{2}}{2t^{2H}}\right)\psi(t)\,\mathrm{d}\tau
\end{equation} 
is $\mathrm{d}t$-integrable on $[0,T]$. The details on the integrability of the function in \eqref{eq:step2} are given in Lemma~\ref{lem:proof-lemma-step2}, below. 
If we integrate both sides of \eqref{eq:Fprime-divided} from $t=0$ to $t=T$ we obtain
\begin{multline}
S(\xi(x))(\varphi)=\frac{1}{\sqrt{2\pi}}\int_{0}^{T}\frac{F'(t)}{t^{H}}\,\mathrm{d}t \\ -\frac{H}{\sqrt{2\pi}}\int_{0}^{T}\int_{0}^{1}\frac{(x-\tau\psi(t))^{2}}{t^{3H+1}}\exp\left(-\frac{(x-\tau\psi(t))^{2}}{2t^{2H}}\right)\psi(t)\,\mathrm{d}\tau\,\mathrm{d}t.\label{eq:S-transform-rewritten}
\end{multline}
\textit{Step 3}: An integration by parts shows that
\begin{equation}\label{eq:IbP}
\begin{aligned}
\int_{0}^{T}\frac{F'(t)}{t^{H}}\,\mathrm{d}t & =\frac{1}{T^{H}}\int_{0}^{1}\exp\left(-\frac{(x-\tau\psi(T))^{2}}{2T^{2H}}\right)\psi(T)\,\mathrm{d}\tau\\
 & +H\int_{0}^{T}\frac{1}{t^{H+1}}\int_{0}^{1}\exp\left(-\frac{(x-\tau\psi(t))^{2}}{2t^{2H}}\right)\psi(t)\,\mathrm{d}\tau\,\mathrm{d}t.
\end{aligned}
\end{equation}
We refer to Lemma~\ref{lem:IbP}, below, for details on the integration by parts in Equation~\eqref{eq:IbP}. The statement of Lemma \ref{lem:newStranslem} now follows from \eqref{eq:S-transform-rewritten} combined with \textit{Step 3}.
\end{proof}
\begin{prop}\label{prop:S-Transf-Lasterm-U-funct}
For any $x\in\mathbb{R}$ and $H\in(0,1)$, the expression of the $S$-transform of $\xi(x)$
given in \eqref{eq:newStrans} is a $U$-functional.
\end{prop}

\begin{proof}The proof is given in Appendix~\ref{app:prop}.
\end{proof}

As a consequence of the characterization theorem, we now immediately obtain the following corollary.

\begin{cor}
\label{cor:extension}For any $x\in\mathbb{R}$ and $H\in(0,1)$, the expression of the $S$-transform of $\xi(x)$
given in \eqref{eq:newStrans} extends uniquely to an entire function
on $S_{1,\mathbb{C}}$; see \cite{KLPSW96}.
\end{cor}
We are now ready to present the regularity result.
\begin{thm}
\label{thm:main-theorem} Let $x\neq0$ and $H \in (0,1)$ be given. Then the stochastic current $\xi(x)$ of $1$-dimensional fBm is an element of $\mathcal{G}_{-s}$ for all $s>0$.
\end{thm}

\begin{proof}
We have to show the condition of Theorem~\ref{thm:charact-regular},
that is, 
\begin{equation}
\sup_{P\in\mathbb{P}}\int_{S'_{1,\mathbb{C}}}|S(\xi(x))(\varepsilon Pz)|^2\,\mathrm{d}\nu(z)<\infty,\label{eq:goal}
\end{equation}
for all $0<\varepsilon<1$. To this end, we use the representation
of $S(\xi(x))$ on $S_{1,\mathbb{C}}$ provided by Corollary~\ref{cor:extension}
and show the $\nu$-integrability.

Let $0<\varepsilon<1$ be fixed and $P\in\mathbb{P}$ be given. It follows
from \eqref{eq:newStrans} and the property \eqref{eq:projprop} that for every
$z\in S'_{1,\mathbb{C}}$ 
\begin{eqnarray}\label{eq:S-transform-1}
 &&S(\xi(x))(\varepsilon Pz)\nonumber \\
  &=&\frac{H}{(2\pi)^{1/2}}\int_{0}^{T}\int_{0}^{1}\frac{1}{t^{H+1}}\exp\left(-\frac{(x-\tau\varepsilon\langle P\eta_{t},z\rangle)^{2}}{2t^{2H}}\right)\varepsilon\langle P\eta_{t},z\rangle\,\mathrm{d}\tau\,\mathrm{d}t\nonumber\\
  &+&\frac{1}{\sqrt{2\pi}T^{H}}\int_{0}^{1}\exp\left(-\frac{(x-\tau\varepsilon\langle P\eta_{T},z\rangle)^{2}}{2T^{2H}}\right)\varepsilon\langle P\eta_{T},z\rangle\,\mathrm{d}\tau \\
  &-&\frac{H}{\sqrt{2\pi}}\!\int_{0}^{T}\!\!\!\int_{0}^{1}\frac{(x-\tau\varepsilon\langle P\eta_{t},z\rangle)^{2}}{t^{3H+1}}\exp\left(\!-\frac{(x-\tau\varepsilon\langle P\eta_{t},z\rangle)^{2}}{2t^{2H}}\right)\!\varepsilon\langle P\eta_{t},z\rangle\,\mathrm{d}\tau\,\mathrm{d}t.\nonumber
\end{eqnarray}
We show the $\nu$-integrability of the last summand in \eqref{eq:S-transform-1}.
All the other summands are particular cases, and the same arguments
apply. Define 
\begin{equation}
G(x,\varphi,t,\tau):=\left(x-\tau\langle\varphi,\eta_{t}\rangle\right)^{2}\exp\left(-\frac{\left(x-\tau\langle\varphi,\eta_{t}\rangle\right)^{2}}{2t^{2H}}\right)\langle\varphi,\eta_{t}\rangle,\label{eq:auxiliar-function}
\end{equation}
where $\varphi\in S_{1,\mathbb{C}}$, $t\in(0,T]$, and $\tau\in[0,1]$.
The following estimate follows from Jensen's inequality 
\begin{eqnarray}
 &  & \int_{S'_{1,\mathbb{C}}}\left|\int_{0}^{T}\int_{0}^{1}\frac{1}{t^{3H+1}}G(x,\varepsilon Pz,t,\tau)\,\mathrm{d}\tau\,\mathrm{d}t\right|^{2}\mathrm{d}\nu(z)\nonumber \\
 & \leq & T\int_{S'_{1,\mathbb{C}}}\int_{0}^{T}\left|\int_{0}^{1}\frac{1}{t^{3H+1}}G(x,\varepsilon Pz,t,\tau)\ d\tau\right|^{2}\mathrm{d}t\,\mathrm{d}\nu(z)\label{eq:firstestimate}\\
 & = & T \int_{S'_{1,\mathbb{C}}}\int_{0}^{T}\int_{0}^{1}\int_{0}^{1}\frac{1}{t^{6H+2}}G(x,\varepsilon Pz,t,\tau_{1})\overline{G(x,\varepsilon Pz,t,\tau_{2})}\,\mathrm{d}\tau_{1}\,\mathrm{d}\tau_{2}\,\mathrm{d}t\,\mathrm{d}\nu(z).\nonumber 
\end{eqnarray}
From the definition of the function $G$, it follows that 
\[
S'_{1,\mathbb{C}}\times(0,T]\times[0,1]\ni(z,t,\tau)\mapsto G(x,\varepsilon Pz,t,\tau)\in\mathbb{C}
\]
is measurable. Apply Fubini's theorem, and we have to evaluate the $\nu$-integral
\begin{equation}
\int_{S'_{1,\mathbb{C}}}G(x,\varepsilon Pz,t,\tau_{1})\overline{G(x,\varepsilon Pz,t,\tau_{2})}\, \mathrm{d}\nu(z)\label{eq:nu-integral}
\end{equation}
for fixed $t\in(0,T]$, $\tau_{1},\tau_{2}\in[0,1]$. Our aim is
to apply Lemma~\ref{lem:momgenfct} to compute the integral in~\eqref{eq:nu-integral}.
Hence, we split the integrand into an exponential factor and a polynomial
factor.

Let $t\in(0,T]$ be such that $P\eta_{t}\neq0$, otherwise the integral
in \eqref{eq:nu-integral} is zero. For every $z\in S'_{1,\mathbb{C}}$ the exponential part of the integrand $G(x,\varepsilon Pz,t,\tau_{1})\cdot\overline{G(x,\varepsilon Pz,t,\tau_{2})}$
is given by
\begin{eqnarray*}
 &  & \exp\left(-\frac{\left(x-\tau_{1}\varepsilon\langle P\eta_{t},z\rangle\right)^{2}}{2t^{2H}}\right)\exp\left(-\frac{\left(x-\tau_{2}\varepsilon\langle P\eta_{t},\overline{z}\rangle\right)^{2}}{2t^{2H}}\right)\\
 & = & \exp\left(-\frac{x^{2}}{t^{2H}}\right)\exp\left(\frac{x(\tau_{1}+\tau_{2})\varepsilon\langle P\eta_{t},\mathfrak{R}z\rangle}{t^{2H}}+\frac{ix(\tau_{1}-\tau_{2})\varepsilon\langle P\eta_{t},\mathfrak{I}z\rangle}{t^{2H}}\right)\\
 &  & \times\exp\left(-\frac{\tau_{1}^{2}\varepsilon^{2}\langle P\eta_{t},z\rangle^{2}}{2t^{2H}}-\frac{\tau_{2}^{2}\varepsilon^{2}\langle P\eta_{t},\overline{z}\rangle^{2}}{2t^{2H}}\right).
\end{eqnarray*}
The polynomial part of $G(x,\varepsilon Pz,t,\tau_{1})\overline{G(x,\varepsilon Pz,t,\tau_{2})}$
gives
\begin{align*}
 & \big(x-\tau_{1}\varepsilon\langle P\eta_{t},z\rangle\big)^{2}\varepsilon\langle P\eta_{t},z\rangle\big(x-\tau_{2}\varepsilon\langle P\eta_{t},e\overline{z}\rangle\big)^{2}\varepsilon\langle P\eta_{t},e\overline{z}\rangle\\
= & \varepsilon^{2}\big[\langle P\eta_{t},\mathfrak{R}z\rangle^{2}-(i\langle P\eta_{t},\mathfrak{I}z\rangle)^{2}\big]\big[x-\tau_{1}\varepsilon(\langle P\eta_{t},\mathfrak{R}z\rangle+i\langle P\eta_{t},\mathfrak{I}z\rangle)\big]^{2}\\
 & \times\big[x-\tau_{2}\varepsilon(\langle P\eta_{t},\mathfrak{R}z\rangle-i\langle P\eta_{t},\mathfrak{I}z\rangle)\big]^{2}.
\end{align*}
Expanding the right side of the above equality in a finite sum of terms, such as 
\[
Cx^{k_{1}}\varepsilon^{k_{2}}\tau_{1}^{k_{3}}\tau_{2}^{k_{4}}\langle P\eta_{t},\mathfrak{R}z\rangle^{n}(i\langle P\eta_{t},\mathfrak{I}z\rangle)^{m}
\]
with $C\in\mathbb{Z}$ and $0\leq k_{j},n,m\leq6$, $j=1,2,3,4$.
Therefore, the integral in \eqref{eq:nu-integral} simplifies to a
finite sum of integrals of the form
\begin{equation}
\begin{aligned}
 & Cx^{k_{1}}\varepsilon^{k_{2}}\tau_{1}^{k_{3}}\tau_{2}^{k_{4}}\exp\left(-\frac{x^{2}}{t^{2H}}\right)\\
&\times\int_{S'_{1,\mathbb{C}}}\exp\left(\frac{x(\tau_{1}+\tau_{2})\varepsilon\langle P\eta_{t},\mathfrak{R}z\rangle}{t^{2H}}+\frac{ix(\tau_{1}-\tau_{2})\varepsilon\langle P\eta_{t},\mathfrak{I}z\rangle}{t^{2H}}\right) \\
 & \times\exp\left(-\frac{\tau_{1}^{2}\varepsilon^{2}\langle P\eta_{t},z\rangle^{2}}{2t^{2H}}-\frac{\tau_{2}^{2}\varepsilon^{2}\langle P\eta_{t},\overline{z}\rangle^{2}}{2t^{2H}}\right)\langle P\eta_{t},\mathfrak{R}z\rangle^{n}(i\langle P\eta_{t},\mathfrak{I}z\rangle)^{m}\,\mathrm{d}\nu(z).
\end{aligned}
\label{eq:nu-integral-1}
\end{equation}
Note that the image of the measure $\nu$ under the map
\[
(S'_{1,\mathbb{C}},\nu)\longrightarrow\mathbb{R}^{2},\;z\mapsto(\langle\varphi,\mathfrak{R}z\rangle,\langle\phi,\mathfrak{I}z\rangle),
\]
for $L^{2}(\mathbb{R})$-normalized test functions $\varphi,\phi\in S_{1}$, is
the centered Gaussian measure with covariance matrix $\frac{1}{2}\mathrm{Id}_{2}$.
We introduce the following notation.
\begin{equation}
\begin{aligned}c_{P}:=c_{P}(t) & :=\frac{\|P\eta_{t}\|_{2}}{t^{H}},\\
\alpha:=\alpha_{P,\varepsilon}(t,\tau_{1},\tau_{2}) & :=\frac{x(\tau_{1}+\tau_{2})\varepsilon c_{P}}{t^{H}},\\
\beta:=\beta_{P,\varepsilon}(t,\tau_{1},\tau_{2}) & :=\frac{x(\tau_{1}-\tau_{2})\varepsilon c_{P}}{t^{H}},\\
\gamma:=\gamma_{P,\varepsilon}(t,\tau_{1},\tau_{2}) & :=-(\tau_{1}^{2}-\tau_{2}^{2})\varepsilon^{2}c_{P}^{2},\\
\sigma:=\sigma_{P,\varepsilon}(t,\tau_{1},\tau_{2}) & :=\frac{2}{2+(\tau_{1}^{2}+\tau_{2}^{2})\varepsilon^{2}c_{P}^{2}},\\
\sigma^{\prime}:=\sigma_{P,\varepsilon}^{\prime}(t,\tau_{1},\tau_{2}) & :=\frac{2}{2-(\tau_{1}^{2}+\tau_{2}^{2})\varepsilon^{2}c_{P}^{2}}.
\end{aligned}
\label{eq:constants}
\end{equation}
Evaluating the integral \eqref{eq:nu-integral-1} using the above functions 
yields
\begin{align*}
 & \frac{1}{\pi}Cx^{k_{1}}\varepsilon^{k_{2}}\tau_{1}^{k_{3}}\tau_{2}^{k_{4}}\exp\left(-\frac{x^{2}}{t^{2H}}\right)\|P\eta_{t}\|_{2}^{n+m}\int_{\mathbb{R}^{2}}\exp\big(\alpha u+\mathrm{i}\beta v+\mathrm{i}\gamma uv\big)\\
 & \times\exp\left(-\frac{(\tau_{1}^{2}+\tau_{2}^{2})\varepsilon^{2}c_{P}^{2}(u^{2}-v^{2})}{2}-(u^{2}+v^{2})\right)u^{n}(\mathrm{i}v)^{m}\,\mathrm{d}u\,\mathrm{d}v\\
= & \frac{1}{\pi}Cx^{k_{1}}\varepsilon^{k_{2}}\tau_{1}^{k_{3}}\tau_{2}^{k_{4}}\exp\left(-\frac{x^{2}}{t^{2H}}\right)\|P\eta_{t}\|_{2}^{n+m}\\
 & \times\int_{\mathbb{R}^{2}}\exp\left(-\frac{1}{\sigma}u^{2}+\alpha u\right)\exp\left(-\frac{1}{\sigma^{\prime}}v^{2}+i\beta v\right)\exp(\mathrm{i}\gamma uv)u^{n}(\mathrm{i}v)^{m}\,\mathrm{d}u\,\mathrm{d}v\\
= & \frac{1}{\pi}Cx^{k_{1}}\varepsilon^{k_{2}}\tau_{1}^{k_{3}}\tau_{2}^{k_{4}}\exp\left(-\frac{x^{2}}{t^{2H}}\right)\|P\eta_{t}\|_{2}^{n+m}\int_{\mathbb{R}^{2}}h(\alpha,\beta,u,v)u^{n}(\mathrm{i}v)^{m}\,\mathrm{d}u\,\mathrm{d}v,
\end{align*}
where $h(\alpha,\beta,u,v)$ is given by
\[ h(\alpha,\beta,u,v):=\exp\left(-\frac{1}{\sigma}u^{2}+\alpha u\right)\exp\left(-\frac{1}{\sigma^{\prime}}v^{2}+i\beta v\right)\exp(\mathrm{i}\gamma uv). \]
Define $I_h$ by
\[ I_{h}:\mathbb{R}^{2}\longrightarrow\mathbb{R},\;(\alpha,\beta)\mapsto\int_{\mathbb{R}^{2}}h(\alpha,\beta,u,v)\,\mathrm{d}u\,\mathrm{d}v .\]
It follows from Lemma~\ref{lem:momgenfct} and Remark~\ref{rem:explicitder} that
\[
\int_{\mathbb{R}^{2}}h(\alpha,\beta,u,v)u^{n}(\mathrm{i}v)^{m}\,\mathrm{d}u\,\mathrm{d}v=\frac{\partial^{n+m}I_{h}}{\partial\alpha^{n}\partial\beta^{m}}(\alpha,\beta)=I_{h}(\alpha,\beta)p_{n+m}(\alpha,\beta).
\]
 Here, $p_{n+m}(\alpha,\beta)$ is a finite sum of terms of the form
\begin{align*}
Dc^{l_{1}}{c^{\prime}}^{l_{2}}(\sigma^{\prime}\gamma)^{l_{3}}\alpha^{l_{4}}\beta^{l_{5}}
\end{align*}
with $D\in\mathbb{Q}$ and $0\leq l_{1},\ldots,l_{5}\leq n+m$. Here, $c$ and $c'$ are defined in \eqref{eq:constants-c}. Recall
that $h,I_{h},c$ and $c^{\prime}$ depend on $\sigma,\sigma^{\prime},\gamma$,
which are functions of $t,\tau_{1},\tau_{2}$. Putting all together,
the integral in \eqref{eq:nu-integral} is given as a finite sum of terms
of the form
\begin{equation}
\frac{1}{\pi}CDx^{k_{1}}\varepsilon^{k_{2}}\tau_{1}^{k_{3}}\tau_{2}^{k_{4}}\exp\left(-\frac{x^{2}}{t^{2H}}\right)\|P\eta_{t}\|_{2}^{n+m}c^{l_{1}}{c^{\prime}}^{l_{2}}(\sigma^{\prime}\gamma)^{l_{3}}\alpha^{l_{4}}\beta^{l_{5}}I_{h}(\alpha,\beta).\label{eq:integral2}
\end{equation}
It remains to show that \eqref{eq:integral2} multiplied by $\frac{1}{t^{6H+2}}$
is integrable with respect to $\mathrm{d}t\otimes\mathrm{d}\tau_{1}\otimes\mathrm{d}\tau_{2}$
on $(0,T]\times[0,1]^{2}$ for any $P\in\mathbb{P}$, and that the
supremum of the resulting set of integrals over all $P\in\mathbb{P}$
is finite. 

First, we note that $I_{h}(\alpha,\beta)$ can be estimated by
\[
I_{h}(\alpha,\beta)\le\frac{\pi}{\sqrt{1-\varepsilon^{2}}}\exp\left(\varepsilon^{2}\frac{x^{2}}{t^{2H}}\right),
\]
see Lemma~\ref{rem:Ih-bound}. Hence, using \eqref{eq:Ih-bound} and bounding the functions in \eqref{eq:constants}, the expression in \eqref{eq:integral2} multiplied by $\frac{1}{t^{6H+2}}$ 
is bounded by
\[
\frac{2^{l_{4}}|CD|T^{H(n+m)}}{4^{l_{1}}\left( 1-\varepsilon^{2}\right)^{l_3+1/2}}|x|^{k_{1}+l_{4}+l_{5}}\left(\frac{1}{4(1-\varepsilon^{2})}+\frac{1}{16(1-\varepsilon^{2})^{2}}\right)^{l_{2}}\frac{\mathrm{e}^{(\varepsilon^{2}-1)\frac{x^{2}}{t^{2H}}}}{t^{6H+2+H(l_{4}+l_{5})}}.
\]
This bound is independent of $\tau_{1},\tau_{2}$ and $P$. Furthermore,
since $\varepsilon^{2}-1<0$, the function 
\[
(0,T]\ni t\mapsto\frac{\mathrm{e}^{(\varepsilon^{2}-1)\frac{x^{2}}{t^{2H}}}}{t^{6H+2+H(l_{4}+l_{5})}}
\]
is $\mathrm{d}t$-integrable. The result of the theorem follows.
\end{proof}

\subsection*{Acknowledgments}

This work was partially supported by a grant from the Niels Henrik
Abel Board and the Center for Research in Mathematics and Applications
(CIMA) related to Statistics, Stochastic Processes, and Applications
(SSPA) group, through the grant UID/4674/2025 of FCT-Funda{\c c\~a}o para a Ci{\^e}ncia e a Tecnologia, Portugal.

\appendix
\section*{Appendix}
\renewcommand{\thesection}{A}
\section{Details on the proof of Lemma \ref{lem:newStranslem}}

\label{sec:appendix}

\begin{lem}\label{lem:proof-lemma-step2}
    The function in Equation~\eqref{eq:step2} is integrable on $[0,T]$, $T>0$.
\end{lem}
\begin{proof}
Let us estimate the function
\begin{equation}
t \mapsto \frac{(x-\tau\psi(t))^{2}}{t^{2H+1}}\exp\left(-\frac{(x-\tau\psi(t))^{2}}{2t^{2H}}\right)\psi(t).\label{eq:integrand}
\end{equation}
First, note that $|\psi(t)|=|\langle\varphi,M_{-}^{H}\mathbbm{1}_{[0,t)}\rangle|\le t\|M_{+}^{H}\varphi\|_{\infty}$ and $|\psi(t)|^{2}\le t^{2H}\|\varphi\|_{2}^{2}$.
The first factor in \eqref{eq:integrand} with $\tau\in[0,1]$ can
be estimated as
\[
\frac{(x-\tau\psi(t))^{2}}{t^{2H+1}}\le\frac{|x|^{2}}{t^{2H+1}}+\frac{2|x|\|M_{+}^{H}\varphi\|_{\infty}}{t^{2H}}+\frac{\|M_{+}^{H}\varphi\|_{\infty}^{2}}{t^{2H-1}}.
\]
The exponential factor in \eqref{eq:integrand} is estimated by 
\begin{align*}
\exp\left(-\frac{(x-\tau\psi(t))^{2}}{2t^{2H}}\right) & \le\exp\left(-\frac{|x|^{2}}{2t^{2H}}\right)\exp\left(\frac{\tau^{2}t^{2H}\|\varphi\|_{2}^{2}}{2t^{2H}}+\frac{|x|\tau |\psi(t)|}{t^{2H}}\right)\\
 & \le\exp\left(-\frac{|x|^{2}}{2t^{2H}}\right)\exp\left(\frac{\|\varphi\|_{2}^{2}}{2}+\frac{|x|^2}{4t^{2H}}+\frac{t^{2H}\|\varphi\|_{2}^{2}}{t^{2H}}\right).
\end{align*}
Putting everything together, the integrand function \eqref{eq:integrand} is
bounded by
\begin{multline*}
\left|\frac{(x-\tau\psi(t))^{2}}{t^{2H+1}}\exp\left(-\frac{(x-\tau\psi(t))^{2}}{2t^{2H}}\right)\psi(t)\right|\\
\le\frac{1}{t^{H}}\left(\frac{|x|^{2}}{t^{H}}+\frac{2|x|\|M_{+}^{H}\varphi\|_{\infty}}{t^{H-1}}+\frac{\|M_{+}^{H}\varphi\|_{\infty}^{2}}{t^{H-2}}\right)\\
\times\exp\left(-\frac{|x|^{2}}{2t^{2H}}\right)\exp\left(\frac{3\|\varphi\|_{2}^{2}}{2}+\frac{|x|^2}{4t^{2H}}\right)\|M_{+}^{H}\varphi\|_{\infty}.
\end{multline*}
Using the fact that $t\in[0,T]$ we further estimate the integrand as
\begin{multline*}
  \le\frac{1}{t^{2H}}\left(|x|^{2}+2T|x|\|M_{+}^{H}\varphi\|_{\infty}+T^2\|M_{+}^{H}\varphi\|_{\infty}^{2}\right)\\
  \times\exp\left(\frac{3\|\varphi\|_{2}^{2}}{2}\right)\exp\left(-\frac{|x|^{2}}{4t^{2H}}\right)\|M_{+}^{H}\varphi\|_{\infty}.
\end{multline*}
Note that the map $(0,T]\ni t\mapsto t^{-2H}\exp(-|x|^{2}/(4t^{2H}))$,
is integrable on $[0,T]$. Then, we may conclude that the double integral
\[
\frac{H}{\sqrt{2\pi}}\int_{0}^{T}\int_{0}^{1}\frac{(x-\tau\langle\varphi,\eta_{t}\rangle)^{2}}{t^{2H+1}}\mathrm{e}^{-\frac{(x-\tau\langle\varphi,\eta_{t}\rangle)^{2}}{2t^{2H}}}\langle\varphi,\eta_{t}\rangle\,\mathrm{d}\tau\,\mathrm{d}t
\]
is finite.    
\end{proof}
\begin{lem}\label{lem:IbP}
    Integration by parts in Equation~\eqref{eq:IbP}.
\end{lem}
\begin{proof}
Given that both functions $F$ and the mapping $(\varepsilon,T)\ni t\mapsto\frac{1}{t^{H}}$ are continuously differentiable on the interval $(\varepsilon,T)$, $\varepsilon>0$, applying integration by parts results in
\begin{align*}
\int_{\varepsilon}^{T}\frac{F'(t)}{t^{H}}\,\mathrm{d}t 
& =\frac{1}{T^{H}}\int_{0}^{1}\exp\left(-\frac{(x-\tau\psi(T))^{2}}{2T^{2H}}\right)\psi(T)\,\mathrm{d}\tau\\
& -\frac{1}{\varepsilon^{H}}\int_{0}^{1}\exp\left(-\frac{(x-\tau\psi(\varepsilon))^{2}}{2\varepsilon^{2H}}\right)\psi(\varepsilon)\,\mathrm{d}\tau\\
& +H\int_{\varepsilon}^{T}\frac{1}{t^{H+1}}\int_{0}^{1}\exp\left(-\frac{(x-\tau\psi(t))^{2}}{2t^{2H}}\right)\psi(t)\,\mathrm{d}\tau\,\mathrm{d}t.
\end{align*}
It is easy to obtain the estimates:
\[
\left|-\frac{1}{\varepsilon^{H}}\int_{0}^{1}\exp\left(-\frac{(x-\tau\psi(\varepsilon))^{2}}{2\varepsilon^{2H}}\right)\psi(\varepsilon)\,\mathrm{d}\tau\right|\le\|M_{+}^{H}\varphi\|_{\infty}\varepsilon^{1-H}
\]
and 
\[
\left|\frac{1}{t^{H+1}}\int_{0}^{1}\exp\left(-\frac{(x-\tau\psi(t))^{2}}{2t^{2H}}\right)\psi(t)\,\mathrm{d}\tau\right|\le\|M_{+}^{H}\varphi\|_{\infty}\frac{1}{t^{H}}.
\]
As the function $(0,T]\ni t\mapsto\|M_{+}^{H}\varphi\|_{\infty}\frac{1}{t^{H}}$
is integrable, then an application of the Lebesgue dominated convergence theorem as $\varepsilon\to0$ yields
\begin{align*}
\int_{0}^{T}\frac{F'(t)}{t^{H}}\,\mathrm{d}t & =\frac{1}{T^{H}}\int_{0}^{1}\exp\left(-\frac{(x-\tau\psi(T))^{2}}{2T^{2H}}\right)\psi(T)\,\mathrm{d}\tau\\
 & +H\int_{0}^{T}\frac{1}{t^{H+1}}\int_{0}^{1}\exp\left(-\frac{(x-\tau\psi(t))^{2}}{2t^{2H}}\right)\psi(t)\,\mathrm{d}\tau\,\mathrm{d}t.
\end{align*}
This shows the claim.
    \end{proof}

\renewcommand{\thesection}{B}
\section{Proof of Proposition~\ref{prop:S-Transf-Lasterm-U-funct}}
\label{app:prop}
\begin{proof}We have to show that for any $x\in\mathbb{R}$ and $H \in (0,1)$ the expression of the $S$-transform of $\xi(x)$ given in \eqref{eq:newStrans} satisfies the conditions of Definition~\ref{def:U-functional}.
Consider the measure space $\big((0,T]\times[0,1],\mathcal{B}((0,T])\otimes\mathcal{B}([0,1])),\mathrm{d}t\otimes\mathrm{d}\tau\big)$
and fix $(t,\tau)\in(0,T]\times[0,1]$. We consider only the last summand since it is the most intricate one, and the other two can be estimated similarly. The corresponding functional to be integrated is denoted by
\[
S_{1}\ni\varphi\mapsto U_{t,\tau}(\varphi):=\frac{(x-\tau\langle\varphi,\eta_{t}\rangle)^{2}}{t^{3H+1}}\exp\left(-\frac{(x-\tau\langle\varphi,\eta_{t}\rangle)^{2}}{2t^{2H}}\right)\langle\varphi,\eta_{t}\rangle.
\]
Let us check the first condition of Definition~\ref{def:U-functional}. Given $\varphi,\theta\in S_{1}$
we have to show that the function 
\[
\mathbb{R}\ni y\mapsto U_{t,\tau}(\varphi+y\theta)\in\mathbb{C}
\]
extends to an entire function on $\mathbb{C}$. This follows easily
from the fact that the function $\mathbb{R}\ni y\mapsto\langle\varphi,\eta_{t}\rangle+y\langle\theta,\eta_{t}\rangle$
is entire, and the composition of entire functions is again an entire
function. Next, we prove the second condition of Definition~\ref{def:U-functional}. Let
$z\in\mathbb{C}$ be given. The following estimate holds
\begin{align*}
& |U_{t,\tau}(z\varphi)| \\
& =\left|\langle z\varphi,\eta_{t}\rangle\frac{(x-\tau\langle z\varphi,\eta_{t}\rangle)^{2}}{t^{2H+1}}\exp\left(-\frac{(x-\tau\langle z\varphi,\eta_{t}\rangle)^{2}}{2t^{2H}}\right)\right|\\
 & \le\frac{|z|\|\varphi\|_2}{t^{2H+1}}\big(|x|+\tau|z|t\|M_{+}^{H}\varphi\|_{\infty}\big)^{2}\exp\left(-\frac{x^{2}-2\tau x\langle z\varphi,\eta_{t}\rangle+\tau^{2}\langle z\varphi,\eta_{t}\rangle^{2}}{2t^{2H}}\right)\\
 & \le\frac{2}{t^{2H+1}}\big(|x|^{2}+t^{2}|z|^{2}\|M_{+}^{H}\varphi\|_{\infty}^{2}\big)\exp\left(-\frac{|x|^{2}}{4t^{2H}}+\frac{3}{2}|z|^{2}\|\varphi\|_{2}^{2}\right)\exp\big(|z|\|\varphi\|_{2}\big)\\
 & =\frac{2}{t^{2H+1}}\exp\left(|x|^{2}+t^{2}|z|^{2}\|M_{+}^{H}\varphi\|_{\infty}^{2}-\frac{|x|^{2}}{4t^{2H}}+\frac{5}{2}|z|^{2}\|\varphi\|_{2}^{2}\right)\\
 & =\frac{2}{t^{2H+1}}\exp\left(|x|^{2}-\frac{|x|^{2}}{4t^{2H}}\right)\exp\left(T^{2}|z|^{2}\|M_{+}^{H}\varphi\|_{\infty}^{2}+\frac{5}{2}|z|^{2}\|\varphi\|_{2}^{2}\right)\\
 & \le C(t, \tau, x,H)\exp\left(\left(T^{2}+\frac{5}{2}\right)|z|^{2}\|\varphi\|^{2}\right),
\end{align*}
 where $\|\cdot\|$ is a continuous norm on $S_{1}$ and $C(t,\tau,x,H)$ is given by
\[
C(t,\tau,x,H):=\frac{2}{t^{2H+1} }\exp\left(|x|^{2}-\frac{|x|^{2}}{4t^{2H}}\right).
\]
Thus, $U_{t,\tau}$ is a $U$-functional. Moreover, the function $(0,T] \times [0,1] \ni (t, \tau) \mapsto C(t,\tau, x,H)$ is $\mathrm{d}t \otimes\mathrm{d}\tau$-integrable, which implies that 
\[
S_{1}\ni\varphi\mapsto\int_{0}^{T}\int_{0}^{1}U_{t,\tau}(\varphi)\,\mathrm{d}\tau\,\mathrm{d}t
\]
is a $U$-functional by Corollary~\ref{cor:hida_integral}.
\end{proof}

\renewcommand{\thesection}{C}
\section{Other Lemmas}
In this concluding appendix, we present several lemmas. The proofs are based on computations involving Gaussian integrals and estimations of basic polynomials. While not challenging, these proofs require substantial space. As such, we leave them to interested readers. Below, the constants $c$ and $c'$ are defined by

\begin{equation}\label{eq:constants-c}
c:=\frac{1}{\frac{4}{\sigma}+\sigma^{\prime}\gamma^{2}},\qquad c^{\prime}:=-\frac{\sigma^{\prime}}{4}+c\frac{{\sigma^{\prime}}^{2}\gamma^{2}}{4}.
\end{equation}
\begin{lem}
\label{lem:momgenfct}Given $\sigma,\sigma^{\prime}>0$ and $\gamma\in\mathbb{R}$, define $h:\mathbb{R}^{4}\longrightarrow\mathbb{C}$ by
\[
h(\alpha,\beta,u,v):=\exp\left(-\frac{1}{\sigma}u^{2}+\alpha u\right)\exp\left(-\frac{1}{\sigma^{\prime}}v^{2}+\mathrm{i}\beta v\right)\exp(\mathrm{i}\gamma uv).
\]
Then, for every $\alpha,\beta\in\mathbb{C}$, $h(\alpha,\beta,\cdot,\cdot)\in L_{\mathbb{C}}^{1}(\mathrm{d}u\otimes\mathrm{d}v)$
and we have 
\[
\int_{\mathbb{R}^{2}}h(\alpha,\beta,u,v)\,\mathrm{d}u\,\mathrm{d}v=2\pi\sqrt{c\sigma^{\prime}}\exp\left(-\frac{\sigma^{\prime}}{4}\beta^{2}+c\big(\alpha-\frac{\sigma^{\prime}\gamma}{2}\beta\big)^{2}\right)\in\mathbb{C}.
\]
Moreover, the function
\[
I_{h}:\mathbb{R}^{2}\longrightarrow\mathbb{C},\;(\alpha,\beta)\mapsto\int_{\mathbb{R}^{2}}h(\alpha,\beta,u,v)\,\mathrm{d}u\,\mathrm{d}v
\]
is in $C^{\infty}(\mathbb{R}^2)$ and the partial derivatives are given
by 
\[
\frac{\partial^{n+m}I_{h}}{\partial\alpha^{n}\partial\beta^{m}}(\alpha,\beta)=\int_{\mathbb{R}^{2}}h(\alpha,\beta,u,v)u^{n}(iv)^{m}\,\mathrm{d}u\,\mathrm{d}v,\qquad n,m\in\mathbb{N}_{0}.
\]
\end{lem}

\begin{rem}
\label{rem:explicitder}The partial derivatives of $I_{h}$ from Lemma~\ref{lem:momgenfct}
may be computed directly. In fact, if we rewrite $I_{h}$ as 
\[
I_{h}(\alpha,\beta)=2\pi\sqrt{c\sigma^{\prime}}\exp\left(c\alpha^{2}+c^{\prime}\beta^{2}-c\sigma^{\prime}\gamma\alpha\beta\right),\quad\alpha,\beta\in\mathbb{R},
\]
then it is clear that the partial derivatives of $I_{h}$ have the
form 
\[
\frac{\partial^{n+m}I_{h}}{\partial\alpha^{n}\partial\beta^{m}}(\alpha,\beta)=I_{h}(\alpha,\beta)\ p_{n+m}(\alpha,\beta),\quad\alpha,\beta\in\mathbb{R},\qquad n,m\in\mathbb{N}_{0},
\]
with $p_{k},\ k\in\mathbb{N}_{0},$ being a polynomial in two variables
of degree $k$. Note that the coefficients of $p_{n+m}$ contain different
powers of $c,c^{\prime},\sigma^{\prime}$ and $\gamma$. An explicit
formula is proven in Lemma~\ref{lem:explicitform} below.
\end{rem}

\begin{lem}
\label{lem:explicitform}Let $\sigma,\sigma'>0$ and $\gamma\in\mathbb{R}$ be given. Define the function
\[
I_{h}\colon\mathbb{R}^{2}\longrightarrow\mathbb{R},\;(\alpha,\beta)\mapsto2\pi\sqrt{c\sigma^{\prime}}\exp(c\alpha^{2}+c^{\prime}\beta^{2}-c\sigma^{\prime}\gamma\alpha\beta).
\]
Then the partial derivatives of $I_{h}$ are given by the formula
\[
\frac{\partial^{n+m}I_{h}}{\partial\alpha^{n}\partial\beta^{m}}(\alpha,\beta)=I_{h}(\alpha,\beta)p_{n+m}(\alpha,\beta),\qquad n,m\in\mathbb{N}_{0},
\]
with
\begin{align*}
 & p_{n+m}(\alpha,\beta)\\
 &=\sum_{i=0}^{m}\sum_{k_{1}=0}^{i}\sum_{\genfrac{}{}{0pt}{3}{k_{2}=0}{i-k_{1}-k_{2}\ \mathrm{even}}}^{i-k_{1}}\sum_{j=m-i}^{n}\sum_{\genfrac{}{}{0pt}{3}{l=0}{n-j-l\ \mathrm{even}}}^{n-j}\genfrac{(}{)}{0pt}{}{m}{i}\frac{j!}{(j-m+i)!}R\left(i,\frac{i-k_{1}-k_{2}}{2},k_{2}\right)\\
 & \times R\left(n,\frac{n-j-l}{2},l\right)c^{n+k_{1}-\frac{n-j-l}{2}}{c^{\prime}}^{k_{2}+\frac{i-k_{1}-k_{2}}{2}}(\sigma^{\prime}\gamma)^{k_{1}+j}\alpha^{k_{1}+l}\beta^{k_{2}+j-m+i}
\end{align*}
where we set 
\[
R(i,j,k):=\frac{i!}{j!(i-2j)!}\binom{i-2j}{k}(-1)^{i-2j-k}2^{k},\qquad i,j,k\in\mathbb{N}_{0}.
\]
\end{lem}

\begin{lem}
\label{rem:Ih-bound}For any $0\le\tau_{i}\le1$, $i=1,2$, and $\alpha,\beta,\sigma,\sigma'$ as in Equation~\ref{eq:constants}, the following bound holds
\begin{equation}
I_{h}(\alpha,\beta)=2\pi\sqrt{c\sigma^{\prime}}\exp(c\alpha^{2}+c^{\prime}\beta^{2}-c\sigma^{\prime}\gamma\alpha\beta)\le\frac{\pi}{\sqrt{1-\varepsilon^{2}}}\exp\left(\varepsilon^{2}\frac{x^{2}}{t^{2H}}\right).\label{eq:Ih-bound}
\end{equation}
\end{lem}

\begin{lem}
\label{lem:polybound}Let $0<\varepsilon\leq1$ be given. The polynomial $q:\mathbb{R}^{2}\longrightarrow\mathbb{R}$ defined by
\[ q(x,y)=\varepsilon^{2}x^{4}+2\varepsilon^{2}x^{2}+\varepsilon^{2}y^{4}+2\varepsilon^{2}y^{2}-2\varepsilon^{2}x^{3}y-2\varepsilon^{2}xy^{3}-2\varepsilon^{2}x^{2}y^{2}+4xy
\]
is bounded by 4 on $[0,1]^{2}$.
\end{lem}

\bibliographystyle{elsarticle-num}
\bibliography{luis}
\end{document}